\documentclass[psamsfonts]{amsart}

\usepackage{amssymb,amsfonts}
\usepackage[all,arc]{xy}
\usepackage{enumerate}
\usepackage{mathrsfs}
\usepackage{graphicx}
\usepackage{color}
\usepackage{float}
\usepackage{enumitem}
\usepackage{cite}
\setitemize[0]{leftmargin=5pt,itemindent=5pt}
\newtheorem{thm}{Theorem}[section]

\newtheorem{prop}[thm]{Proposition}
\newtheorem{lem}[thm]{Lemma}

\theoremstyle{definition}
\newtheorem{defn}[thm]{Definition}

\theoremstyle{remark}
\newtheorem{rem}[thm]{Remark}

\newcommand{\tab}{\hspace*{2em}}

\makeatletter
\let\c@equation\c@thm
\makeatother
\numberwithin{equation}{section}

\title{A Bound on Length Isospectral Families of Hyperbolic Surfaces}
\author{Weston Ungemach}
\date{\today}

\begin{document}

\begin{abstract}
In this paper we obtain a bound on the number of isometry classes of finite area hyperbolic surfaces length isospectral to a given surface depending only on
the topological type of the surface and the length of the shortest closed geodesic on the surface. This will follow from a more general bound applying to any family of hyperbolic surfaces which admits a Bers' constant and with a lower bound on systole length.
\end{abstract}

\maketitle

\tableofcontents

\section{Introduction}

	Recall that the length spectrum of a compact hyperbolic surface is the ordered collection of lengths of all closed geodesics on the surface, counted with multiplicity, and that two hyperbolic surfaces are length isospectral if they share the same length spectrum. Of particular interest are those surfaces length isospectral to a given surface, but not isometric to it. 
	\begin{defn}
		A length isospectral family is a collection of hyperbolic surfaces which are pair-wise isospectral and non-isometric.
	\end{defn}
	In this paper, all hyperbolic surfaces will be assumed complete and without boundary. It has been shown that for fixed genus $g$ the collection of compact hyperbolic surfaces whose isometry type is determined by their length spectrum forms a generic set in the moduli space of such surfaces \cite{Wolpert1, Wolpert2}. There is, however, a non-empty, local real analytic subvariety of this space for which the spectrum does not determine isometry-type \cite{Wolpert1, Vigneras}. In this case, the collection of all isometry classes of compact hyperbolic surfaces length isospectral to a given surface $S$ has been shown to be finite \cite{McKean} and Buser \cite[Thm 13.1.1]{Buser 2} has produced the following upper bound:
	\begin{thm}
		Let $S$ be a compact hyperbolic surface of genus $g \geq 2$. Then there exist at most $\exp(720g^2)$ pairwise non-isometric compact hyperbolic surfaces length isospectral to $S$.
	\end{thm}
	As the largest length isospectral families constructed to date are on the order of $e^{\log(g)^2}$ \cite{Brooks}, it would seem that the asymptotics on Buser's bound are significantly too large. Our main results are as follows:
	\begin{thm}
			Allow $L_I(g,n)$ to denote the largest size of any length isospectral family in the set of finite area hyperbolic surfaces with genus $g$ and $n$ cusps and systole length at least $I$. Then the following inequalities hold, where $g,n$ are non-zero unless indicated:
		\begin{itemize}
			\item $L_I(g,n) < I^{-4(3g-3+n)}e^{380g^2 + 46n^2 + 257ng}$.\vspace*{1mm}
			\item $L_I(g,0) < I^{12(1-g)}e^{171g^2}.$ \vspace*{1mm}
			\item If $n \geq 3$, then $L_I(0,n) < I^{12-4n}e^{300(n-3)\sqrt{n-2} + 2(n-3)\log(n+3) + n\log(n-1) + 24n}$
			\item $L_I(2,0) < I^{-12} \cdot 3.8 \times 10^{53}$. \vspace*{1mm}
		\end{itemize}
	\end{thm}
	
	In fact, we prove a more general result. Namely,
	
	\begin{thm}
		Let $X$ be a collection of hyperbolic surfaces with genus $g$ and $n$ cusps, Bers' constant $B$, and systole of length at least $I$. Then the maximum size of a length isospectral family in $X$ is bounded above by:	
		\begin{align*}
			L_{B,I}(g,n) &= 2\exp\{2(3g-3 + n)(2B - 2\log(I) + \log(2g-2+n) +\\
			& \hspace*{50mm} 5\log(2) + 6) + \log(C(g,n))\}.
		\end{align*}
		where 
		\begin{equation*}
			C(g,n) = \begin{cases}
				ng^{3g}(3g-3+3n)^n, & \text{if } n \not= 0 \\
				g^{3g}, & \text{if } n = 0
			\end{cases}
		\end{equation*}
	\end{thm}
For convenience, we recall two definitions used in the hypothesis of the previous theorem:
	\begin{defn} \textcolor{white}{stuff.}
		\begin{enumerate}
			\item A family $X$ of hyperbolic surfaces is said to have Bers' constant $B$ if every $S \in X$ admits a pants decomposition $\gamma_1,...,\gamma_m$ such that $l(\gamma_i) \leq B$ for all $i$.
			\item A systole on a hyperbolic surface $S$ is a minimal length, nontrivial closed geodesic on the surface. The systole length of $S$ is the length of any systole of $S$.
		\end{enumerate}
	\end{defn}
	Note that every collection of genus $g$ hyperbolic surfaces with $n$ punctures admits a Bers' constant of $10.13(3g-3+n)$ - which is shown in \cite{Buser2} and reproduced in Proposition 3.1 - so that $B$ can always be taken less than or equal to this bound, depending upon the surfaces in question. As the length spectrum of a surface determines both the genus \cite[Thm 9.2.14]{Buser2} and the number of cusps \cite{Borthwick}, the Bers' constant for surfaces with genus $g$ and $n$ punctures yields a bound on the number of surfaces which could possibly be length isospectral to a surface with that genus and number of cusps - not just those in the collection $X$. Here, the dependence on systole length is given by a multiplicative factor of $I^{-4(3g-3+n)}$ and the dependence on the Bers' constant is a multiplicative factor of $\exp\{4B(3g-3+n)\}$. \\ \tab
	To achieve the numerical values in Theorem 1.3 from the general result Theorem 1.4, we have used the following Bers' constants:
	
		\begin{center}
			\begin{tabular}{ | l | l | p{5cm} |}
			\hline
			Class of Surfaces                 & Bers' Constant        & Citation \\ \hline
			finite area Surfaces              & $10.13(3g-3+n)$       & Proposition 3.1 \\ \hline
			Closed Surfaces                   & $12.67(g-1) + 20$     & \cite[Parlier]{Parlier1}  \\ \hline
			Punctured Spheres                 & $30\sqrt{2\pi(n-2)}$  & \cite[Florent, Parlier]{Parlier2} \\ \hline
			Genus 2 Surfaces                  & 4.45                  & \cite[Gendulphe]{Gendulphe} \\ \hline
			\end{tabular}
		\end{center}
		
	The proof of Theorem 1.4 will proceed by using the structure of Teichmuller space to say that every surface we consider will be isometric to one determined by a set of parameters - a graph encoding a pants decomposition and the length and twist parameters - and then counting these parameters using our assumed bound on systole length and existence of a Bers' constant.\\ \tab
	Allow $S$ to be a fixed surface with $S'$ isospectral to $S$. The number of possible graphs encoding a corresponding pants decomposition will be calculated in Proposition 2.4. In order to bound the number of possible length parameters in $S'$, we will use the assumed Bers' constant to get an upper bound on the possible lengths of the curves defining them, count the number of curves on $S$ of length less than this bound, and then by isospectrality conclude that the length parameters of $S'$ must come from this same finite collection. Here, a key role is played by Proposition 4.3. To bound the number of possible twist parameters on $S'$, we first show in Proposition 3.3 that the twist parameters are essentially determined by the lengths of a particular collection of closed curves on the surface - the transversals to the curves in our pants decomposition; we then calculate a bound on the lengths of these curves and proceed as in the case of the length parameters.\\ \tab
	This final bounding of the twist parameters is where our assumption on systole length will prove essential. It is a consequence of the collar theorem that as a simple closed geodesic shrinks, it has an embedded cylinder of increasing length about it on the surface. This results in the lengthening of the transverse curve to the geodesic. Thus, allowing for arbitrarily short geodesics results in arbitrarily long transversals and the proof presented here is no longer valid.
	
\section{Background and Notation: Pants Decompositions and their Combinatorics}

	We first remark that all material discussed in this section is either presented in or a natural generalization of the material presented in \cite{Buser2}.
		\begin{defn}
			We denote the moduli space of isometry classes of finite area hyperbolic surfaces of genus $g$ with $n$ cusps by $\mathcal{F}(g,n)$.
		\end{defn}
		\begin{defn}
			A pants decomposition on a surface $S \in \mathcal{F}(g,n)$ is a collection of simple closed curves $\gamma_1,...,\gamma_k$ on $S$ such that the complement of these geodesics is homeomorphic to a disjoint union of pairs of pants, i.e. thrice punctured spheres.
		\end{defn}
	It is a well known fact that every closed hyperbolic surface admits a decomposition into pairs of pants and that there are exactly $3g-3$ closed curves in the decomposition for a surface of genus $g$ \cite[Thm 3.6.4]{Buser2}. As there is a unique geodesic in the free homotopy class of each of these curves \cite[Thm 1.6.6]{Buser2} and geodesics minimize intersection \cite[Thm 1.6.7]{Buser2}, these curves may be taken to be geodesics. A choice of a pants decomposition can be combinatorially encoded by three sets of parameters: a 3-regular graph with $2g-2$ nodes, a $(3g-3)$-tuple of real numbers known as the length parameters, and a $(3g-3)$-tuple of real numbers between $-\frac{1}{2}$ and $\frac{1}{2}$ known as the twist parameters. In this section we discuss the 3-regular graphs and the length parameters and regulate a discussion of the twist parameters to the following section.
	\begin{figure}[H]
	\begin{center}
		\includegraphics[width=1.0\textwidth]{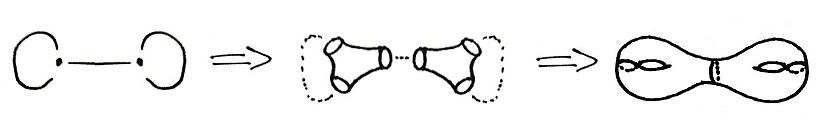} 
	\end{center}
	\caption{}
	\end{figure}
	A 3-regular graph is a graph in which every node has three emanating half-edges (a loop at a node counts as two emanating half-edges). Such a graph encodes a pants decomposition by replacing each node with a pair of pants where the boundary components open along the emanating half-edges from that node. Boundary components along the two half-edges of the same edge are then glued together creating a surface. In order that this glueing be sensible, however, we need the lengths of the boundary components along both halves of each edge to be the same. The choices of these lengths are encoded in the $(3g-3)$-tuple of length parameters. Similarly, any surface in $\mathcal{F}(g,n)$ admits a pants decomposition with $3g-3+n$ curves, and each of these curves has a corresponding geodesic in its free homotopy class. We then have a choice of length and twist parameter for each of these geodesics.
	\begin{figure}[H]
	\begin{center}
		\includegraphics[width=1.0\textwidth]{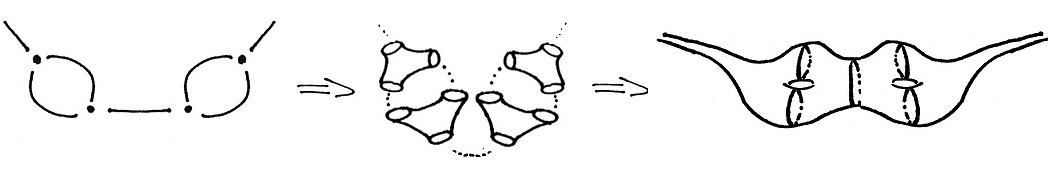} 
		\end{center}
	\caption{}
	\end{figure}
	The final piece of our generalization of pants decompositions to surfaces of finite area is finding a family of graphs which encode the combinatorics of such decompositions. Allow $S \in \mathcal{F}(g,n)$. For each pants decomposition of $S$, we get an associated graph by replacing each pair of pants with a node with three half-edges corresponding to the boundary components and cusps of that pair of pants. Half-edges are connected if they correspond to boundary components which were glued on the surface $S$. This graph will be 3-regular at all but at most $n$ irregular points. These irregular points will still have three emanating half edges, but these edges may not have free ends. An edge with a free end is called a free edge. Free edges correspond to cusps on $S$, as depicted in Figure 2.
	\begin{defn}
		A connected graph with free edges is called pseudo-3-regular if every node has three emanating half-edges, which may or may not have a free end.
	\end{defn}
	\begin{prop}
		If $n \not= 0$, then the number of pseudo-3-regular graphs which correspond to surfaces of genus $g$ with $n$ cusps is bounded above by $ng^{3g}(3g-3+3n)^n$.
	\end{prop}
	
	\begin{proof}
	 Suppose that $G$ is a graph corresponding to a surface of genus $g$ with $n$ cusps. Then we note that $G$ can be constructed by taking a 3-regular graph corresponding to a surface of genus $g$ and adding on the free edges in two possible ways. To add a node corresponding to a pair of pants with one cusp, place a node in the middle of any edge on the graph and add a free edge emanating from that new node. To add a node corresponding to a pair of pants with two cusps, place a node in the middle of any edge on the graph and attach a new edge to it with a new node on the other end; now place two free edges off of this new node. These two moves are shown in Figure 3. To see that any pseudo-3-regular $G$ can be constructed in this way, simply run either or both of these steps in reverse until we obtain a new graph $G'$ which has no free edges. This graph $G'$ will be 3-regular as each of the above moves preserves the 3-regularity of a graph. Thus, any pseudo-3-regular graph which does not correspond to a pair of pants can be constructed in this way. (A pair of pants may have three cusps.)
	\begin{figure}[H]
	\begin{center}
		\includegraphics[width=1.0\textwidth]{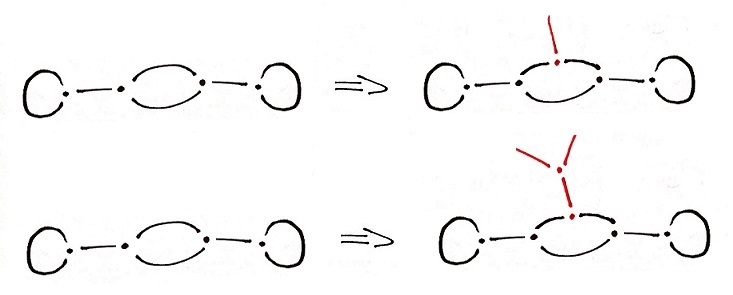} 
	\end{center}
	\caption{}
	\end{figure}
	Now, suppose we are given a 3-regular graph corresponding to a surface of genus $g$ and wish to construct a pseudo-3-regular graph corresponding to a surface of genus $g$ with $n$ cusps. Then we must attach $n$ cusps to this graph, either one or two at a time. As there are at most $n$ ways to partition $n$ as a sum of $1$s and $2$s, there are at most $n$ different sets of one- and two-cusped pairs of pants which will result in a surface with $n$ cusps. Further, each of these collections has at most $n$ elements and as there are at most $3g-3+3n$ edges at any time, they can be attached to the attached to a graph corresponding to a surface of genus $g$ in at most $(3g-3+3n)^n$ ways. As there are at most $g^{3g}$ 3-regular graphs corresponding to a surface of genus $g$ \cite[Thm 3.5.3]{Buser2}, this yields an upper bound of $ng^{3g}(3g-3+3n)^n$ possible pseudo-3-regular graphs corresponding to a surface of genus $g$ with $n$ cusps.
	\end{proof}
	We conclude this section with the definition of a constant to be used in the main result, which will just correspond to the number of graphs corresponding to a surface of a given genus and with given number of cusps:
	\begin{defn}
		We define $C(g,n)$ to be the following upper bound on the number of pseudo-3-regular graphs corresponding to surfaces of genus $g$ with $n$ cusps:
		\begin{equation*}
				C(g,n) = \begin{cases}
					ng^{3g}(3g-3+3n)^n, & \text{if } n \not= 0 \\
					g^{3g}, & \text{if } n = 0
				\end{cases}
		\end{equation*}
	\end{defn}
	
	\section{Background and Notation: Twist Parameters}
		For fixed surfaces $S, S'$ where $S$ has a chosen pants decomposition, we will use a collection of curves called the transverse curves to the pants decomposition on $S$ to obtain an upper bound on the number of possible twist parameters in a given pants decomposition of $S'$. The main tool will be Proposition 3.3. Before proving this fact, we provide a more detailed discussion of the twist parameter and construct this family of transverse curves. \\ \tab
		Suppose we have a standard pair of pants $Y$ - that is, one without cusps. Then for each pair of boundary components of $Y$ there exists a unique common perpendicular between these two components. Collectively, these three common perpendiculars decompose $Y$ into a pasting of two isometric right-angled, geodesic hexagons (as the lengths of three non-adjacent sides of a right-angled, hyperbolic geodesic hexagon determine the hexagon \cite[Thm 2.4.2]{Buser2}). 
		\begin{defn}
			A standard parametrization of the boundary of a pair of pants $Y$ is a parametrization which, for each boundary component, starts at one of the seams where the two hexagons forming $Y$ are glued and proceeds around that boundary component (in either direction) at constant speed and in unit time. 
		\end{defn}
If $Y$ has cusps, than an identical construction exists as there is a unique geodesic from a boundary component out of a cusp so that the previous definition also applied to degenerate pairs of pants. \\ \tab
	Now suppose that we have two pairs of pants $Y,Y'$ with standard parametrizations of their boundary which have boundary components $\gamma$ on $Y$ and $\gamma'$ on $Y'$ such that $l(\gamma) = l(\gamma')$. Then the glueing of $Y$ and $Y'$ along these boundary components with twist parameter $\alpha \in [-\frac{1}{2}, \frac{1}{2}]$ is given by the topological quotient $(Y \sqcup Y')/(\gamma(t) \sim \gamma'(\alpha - t))$ for all $t \in [0,1]$. Different choices of $\alpha$ yield non-isometric surfaces. If we have a pants decomposition $\gamma_1,...,\gamma_l$ on a surface $S$, then the twist parameter at a geodesic $\gamma_i$, then the twist parameter about $\gamma_i$ is the twist parameter corresponding to the glueing of the pairs of pants attached about that geodesic. Note, however, that the very definition of the twist parameter is entirely dependent on our boundary parametrizations and so is entirely non-canonical. 
	\begin{defn}
		 We denote the surface $S \in \mathcal{F}(g,n)$ based on the pseudo-3-regular graph $G$ with length and twist parameters, respectively, $L = (l_1,...,l_{3g-3 + n})$ and $A = (\alpha_1,...,\alpha_{3g-3 + n})$ by $S(G,L,A)$.
	\end{defn}
	Now suppose $S$ is a surface with pants decomposition $\gamma_1,...,\gamma_l$ where these curves $\gamma_i$ are all given the standard parametrization. Then as we have made a choice of parametrization at each boundary component, the twist parameter along each component is well-defined and, up to isometry of $S$, can be taken to lie in $[-\frac{1}{2},\frac{1}{2}]$. Thus, we have obtained a decomposition of our surface $S$ as a surface $S(G,L,A)$. Thus,
	\begin{prop}
		Every element of $\mathcal{F}(g,n)$ is isometric to some $S(G,L,A)$ with $G$ a pseudo-3-regular graph, $L \in \mathbb{R}_{+}^{3g-3+n}$, $A \in [-\frac{1}{2},\frac{1}{2}]^{3g-3+n}$.
	\end{prop}
	Again take $S$ with pants decomposition $\gamma_1,...,\gamma_l$ where these curves $\gamma_i$ are all given the standard parametrization. Then transverse geodesic $\delta_k$ is constructed as follows. \\ \tab
	Consider two pairs of pants $Y,Y'$ in the pants decomposition of $S$ glued together about a geodesic $\gamma_k$ with twist parameter $\alpha$. Allow $Y$ to have boundary components $\gamma_k$, $\xi_1$, and $\xi_2$ and $Y'$ to have boundary components $\gamma_k$, $\eta_1$, and $\eta_2$. Let $p,p'$ be the unique geodesics connecting $\gamma_k$ to $\xi_1, \eta_2$, respectively, oriented away from $\gamma_k$ (see Figure 4). We can the connect $p$ and $p'$ with an arc $c$ lying along $\gamma_k$ with magnitude equal to $|\alpha|l(\gamma_k)$ at $\gamma_k(0)$ in the direction of the twisting, oriented away from $p$. Then the transversal $\delta_k$ to $\gamma_k$ is given by the unique geodesic in the free homotopy class of the concatenated path $\xi_1 p^{-1} c p' \eta_1(p')^{-1} c^{-1} p$. Thus if we have any pair of pants decomposition on a surface $S$, we can isometrically embed the glued pairs of pants $Y$ and $Y'$ into $S$ about any geodesic $\gamma_k$ in the pants decomposition if we choose the appropriate lengths for the boundary geodesics. After embedding, the embedded curve $i(\delta_k)$ is the canonical curve transverse to $\gamma_k$ in $S$. \\ \tab
	For surfaces with cusps, this same construction works after taking a boundary curve to a collar neighborhood of the cusp and treating that as one of the boundary geodesics in the above process.
		\begin{figure}[H]
		\begin{center}
			\includegraphics[width=0.8\textwidth]{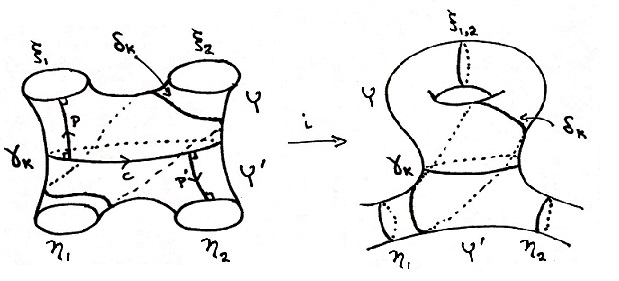} 
		\end{center}
		\caption{}
	\end{figure}
	
	\begin{prop}
		Allow $S$ to be a hyperbolic surface with pants decomposition $\gamma_1,...,\gamma_m$ with the standard parametrization. Then, for each $k$, there are at most two values of the twist parameter $\alpha_k$ at $\gamma_k$ which result in a given value for the length of the transversal $\delta_k$ at $\gamma_k$.
	\end{prop}
	
	\begin{proof}
		Here we prove only the case for closed surfaces as the geometric picture is more tractable. The general case follows directly from the main result in \cite{Bromberg}, which shows that the length of the transversal is a convex function of the twist parameter about a geodesic in a pants decomposition.\\ \tab
		For ease of notation, for any curve $\gamma$ we allow $\text{cosh}(\gamma) := \text{cosh}(l(\gamma))$. We allow the notation of Figure 4. Then we first note that $p^{-1}cp'$ is homotopic with endpoints gliding to some unique geodesic $\phi$ which intersects both $\xi_1$ and $\eta_2$ perpendicularly \cite[Thm 1.6.6(i)]{Buser2}. We claim that the unique curve in the free homotopy class of $\psi = \xi_1 \phi \eta_2 \phi^{-1}$ is disjoint from $\delta_k$. First, cut open the surface $S$ along $\phi$ to create a surface with piecewise geodesic parametrized boundary. By examination of any of the concatenation points of $\psi$, we conclude that the unique geodesic in the free homotopy class of $\psi$ must include non-boundary points. It is a standard result that the unique geodesic in the free homotopy class of a curve with points on the interior of a surface is entirely contained in the surface \cite[Thm 1.6.6(ii)]{Buser2}. As the boundary of this cut-open surface is exactly $\psi$, this says that $\psi$ and $\delta_k$ are disjoint. \\ \tab
		Thus, $\xi_1, \eta_2,$ and $\delta_k$ form the boundary curves of a pair of pants $Z$, and $\phi$ is the unique geodesic orthogonal to both $\xi_1$ and $\eta_2$ in $Z$. Thus, $Z$ is a glueing of two right-angled hyperbolic hexagons of the following form:
		\begin{figure}[H]
		\begin{center}
			\includegraphics[width=0.5\textwidth]{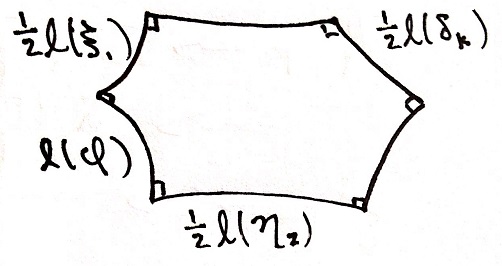} 
		\end{center}
		\caption{}
	\end{figure}
		Then from a relation between the sides of a right-angled, geodesic hexagon \cite[Thm 2.4.1(i)]{Buser2}:
			\begin{equation*}
				\text{cosh}(\frac{1}{2}\delta_k) = \text{sinh}(\frac{1}{2}\xi_1)\text{sinh}(\frac{1}{2}\eta_2)\text{cosh}(\phi) - \text{cosh}(\frac{1}{2}\xi_1)\text{cosh}(\frac{1}{2}\eta_2)
			\end{equation*}
			Now, $\psi = \xi_1 \phi \eta_2 \phi^{-1}$. As the lengths of both $\xi_1$ and $\eta_2$ are fixed, the length of $\psi$ depends only upon the length of $\phi$. We now analyze the dependence of $\phi$ on $\alpha$. Note that $\phi, p, p',$ and $c$ form a crossed right-angled hexagon if we take the remaining sides along $\xi_1$ and $\eta_2$ connecting the edges which intersect these sides. This situation is depicted in Figure 6. 
		\begin{figure}[H]
		\begin{center}
			\includegraphics[width=0.4\textwidth]{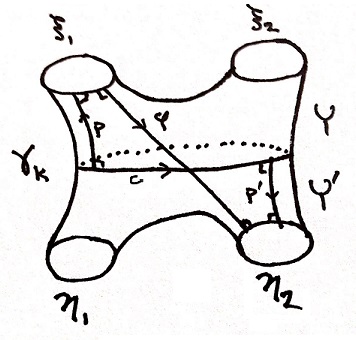} 
		\end{center}
		\caption{}
	\end{figure}
	We can then apply Theorem 2.4.4 from \cite{Buser2} to conclude that:
				\begin{equation*}
					\text{cosh}(\phi) = \text{sinh}(p)\text{sinh}(p')\text{cosh}(c) + \text{cosh}(p)\text{cosh}(p')
				\end{equation*}
	Noting that $l(c) = |\alpha| l(\gamma_k)$, substituting into our previous expression yields:
				\begin{align*}
					\text{cosh}(\frac{1}{2}\delta_k) &= \text{sinh}(\frac{1}{2}\xi_1)\text{sinh}(\frac{1}{2}\eta_2)\{\text{sinh}(p)\text{sinh}(p')\text{cosh}(\alpha l(\gamma_k)) + \text{cosh}(p)\text{cosh}(p')\} \\
					& \hspace{70mm} - \text{cosh}(\frac{1}{2}\xi_1)\text{cosh}(\frac{1}{2}\eta_2)
				\end{align*}
			As the hyperbolic cosine function is 2-to-1, this leaves only two possible choices for the twist parameter.
	\end{proof}
	
	
\section{Bers' Constants and length spectrum growth rate for Surfaces with Cusps}

	\begin{defn}
		A collection of hyperbolic surfaces is said to satisfy length spectrum growth rate $f: \mathbb{R}^+ \rightarrow \mathbb{R}$ if the number of closed geodesics of length $\leq L$ which are not iterates of closed geodesics of length $\leq 2\text{arcsinh}(1)$ is bounded above by $f(L)$ for all $L$.
	\end{defn}
	It is well-known that the class of closed hyperbolic surfaces admits a Bers' constant of $26(g-1)$ and satisfies length spectrum growth rate $f(L) = (g-1)\exp\{L + 6\}$. In this section, we compute similar bounds which apply to arbitrary $S \in \mathcal{F}(g,n)$. The former of these statements will be used in application of the main result and the latter will be essential to prove that result itself. The following is due to Buser:
	\begin{prop}
		The collection $\mathcal{F}(g,n)$ admits a Bers' constant of $10.13(3g - 3 + n)$.
	\end{prop}
		
	\begin{proof}
		It is shown in \cite[Thm 5.2.6]{Buser2} that a hyperbolic surface of this type admits a pants decomposition $\gamma_1,...,\gamma_{3g-3+n}$ such that
			\begin{equation*}
				l(\gamma_k) \leq 4k \log \left(\frac{4\pi(2g - 2 + n)}{k}\right).
			\end{equation*}
		Allowing $k = 3g - 3 + n$ yields the above bound.
	\end{proof}
	
	We now show that $\mathcal{F}(g,n)$ satisfies a certain length spectrum growth rate. The proof proceeds by constructing for each curve of length $\leq L$ a corresponding curve passing through a point in some maximal set of points on the surface with distance bounded pairwise by at least $2\text{arcsinh}(1)$ which lie outside of collar neighborhoods. Outside of the collar, the surface will be thick enough to cover the surface by thick embedded balls and then we can pull into the universal cover and make an area argument. 
	
	\begin{prop}
		$\mathcal{F}(g,n)$ satisfies length spectrum growth rate $f(L) = \frac{1}{2}(2g - 2 + n)\exp\{L + 6\}$.
	\end{prop}
	
	\begin{proof}
		Let $\beta_1,...,\beta_k$ for $k \leq 3g-3+n$ be the simple closed geodesics of length $l(\beta_i) \leq 2 \text{arcsinh}(1)$; the size of this set has the given upper bound as the collar theorem for surfaces with cusps \cite[Thm 4.4.6]{Buser2} says that geodesics of length below $2\text{arcsinh}(1)$ must be simple; then the upper bound on $k$ follows as there are at most $3g-3+n$ disjoint simple closed curves on $S$ (or else we would have a pants decomposition with too many pairs of pants). Let $p_1,...,p_s$ be a maximal set of points on $S$ such that the following two conditions hold:
		\begin{enumerate}
			\item At each $p_\sigma$, the injectivity radius $r_{p_\sigma}(S)$ of $S$ satisfies $r_{p_\sigma} \geq r := \text{arcsinh}(1)$,
			\item $\text{dist}(p_{\sigma}, p_{\tau}) \geq r$ for $\sigma, \tau = 1,...,s, \sigma \not= \tau$.
		\end{enumerate}
	Since disks of radius $r/2$ around $p_1,...,p_s$ are pairwise disjoint and have area equal to $2\pi(\text{cosh}(\frac{r}{2}) - 1)$, and the area of our surface is $2\pi(2g - 2 + n)$, it follows that
		\begin{equation}
			s \leq (2g - 2 + n)/(\text{cosh}(\frac{r}{2} - 1)).
		\end{equation}
	A closed geodesic $\gamma$ of length $> 2r$ on $S$ which is not an iterate of $\beta_{\kappa}$ for $\kappa \in {1,...,k}$, cannot be entirely contained in the collar
		\begin{equation*}
			\mathcal{C}(\beta_{\kappa}) := \{p \in S \mid \text{sinh}(\text{dist}(\beta_{\kappa}, p))\text{sinh}(\frac{1}{2}l(\beta_{\kappa})) \leq 1\},
		\end{equation*}
	as every nontrivial curve in the collar is homotopic to an iterate of $\beta_{\kappa}$. As $\beta_{\kappa}$ is a geodesic, this implies that every geodesic contained in this collar is an iterate of this curve. Now, again by the collar theorem points outside of the collars of the $\beta_i$ have an injectivity radius $\geq r$, so $\gamma$ hits some point $x$ with such an injectivity radius. By the maximality of the set $p_1,...p_s$, this point is a distance $d \leq r$ from some $p_\sigma$. It follows that $c$ is homotopic to a geodesic loop at $p_\sigma$ of length $\leq l(c) + 2r$. \\ \tab
	Now we estimate the number of geodesic loops at $p_\sigma$ of length $\leq L + 2r$. By then accounting for each point $p_{\sigma}$ we will obtain an upper bound on the number of distinct geodesics of length at least $L$. We use a universal covering $\mathbb{H} \rightarrow S$ and lift the loops into geodesic arcs in $\mathbb{H}$ with a common initial point $\bar{p}_{\sigma}$. Since the injectivity radius of $S$ at $p_{\sigma}$ is $r_{p_{\sigma}}(S) \geq r$, the endpoints of these arcs have pairwise distances $\geq 2r$. \\ \tab
	We count the number of possible endpoints by comparing areas. The endpoints of the lifted arcs all have to be within the ball of radius $L + 2r$ about $\bar{p}_{\sigma}$ in order to satisfy the length bound. The condition that the endpoint be separated by at least a distance of $2r$ means that the balls of radius $r$ about the endpoints are disjoint. These balls of radius $r$ are all contained in the ball of radius $L+3r$ as their centers are no further than $L+2r$ from $\bar{p}_{\sigma}$. \\ \tab
	Comparison of the areas of the balls of radius $L+3r$ and radius $r$ yields:
		\begin{equation*}
			\frac{\text{cosh}(L+3r) - 1}{\text{cosh}(r) - 1} \leq \frac{e^{L + 3r}}{2(\text{cosh}(r) - 1)}
		\end{equation*}
	such loops. Noting that $3r < 6$, we multiply by $s$ to account for every possible initial point $p_{\sigma}$ to obtain:
		\begin{equation*}
			s\frac{e^{L + 3r}}{2(\text{cosh}(r) - 1)} \leq \frac{(2g - 2 + n)e^{L+3r}}{2(\text{cosh}(r) - 1)(\text{cosh}(\frac{r}{2} - 1)} \leq \frac{1}{2}(2g - 2 + n)e^{L+3r}
		\end{equation*}
		possible geodesics.
	\end{proof}

\section{Proof of the Main Result}
	We follow the approach in Buser \cite[Chp 13]{Buser2}. Fix a surface $S \in \mathcal{F}(g,n)$ so that $S$ has Bers' constant $B$ and systole length at least $I$. We wish to bound the number of possible length isospectral surfaces with this same Bers' constant and systole length bound. The existence of a pants decomposition and the Bers' constant $B$ establishes the following:
	\begin{prop}
		If $S$ and $S'$ are length isospectral and share the same Bers' constant $B$ and lower bound on systole length $I$, then $S'$ is isometric to a surface $S(G,L,A)$ where $-\frac{1}{2} \leq \alpha_1,...,\alpha_{3g-3+n} \leq \frac{1}{2}$ and $B \geq l(\gamma_1), \dots, l(\gamma_{3g-3+n}) \geq I$.
	\end{prop}
	The bound on the length parameters comes from the definition of the Bers' constant the bound on twist parameters simply comes from our previous discussion of the twist parameters. We now proceed by counting the number of possible graphs, length parameters and twist parameters. \\ \tab
 In Section 2, we bounded the number of graphs which correspond to a surface of genus $g$ with $n$ cusps by $C(g,n)$, where $C(g,n)$ is as defined in Definition 2.5. For the length parameters, the estimate on length spectrum growth rate in Proposition 4.3 implies that there are at most $\frac{1}{2}(2g-2+n)\exp\{B + 6\}$ geodesics of length less than $B$ on the fixed surface $S$. As $S'$ is length isospectral to $S$, we have this same upper bound for the number of geodesics of length less than $B$ on $S'$. As there are $3g-3+n$ lengths in our pants decomposition, we conclude that:
	\begin{lem}
		For each graph $G$ corresponding to $S'$, the number of possible choices of length parameter $L$ on $S'$ is bounded above by $(\frac{1}{2}(2g-2+n))^{3g-3+n}\exp\{(3g - 3 + n)(B + 6)\}$.
	\end{lem}
All that is left to count are the twist parameters. This is facilitated by our bound $I$ on the systole length of $S$, which is equivalent placing a lower bound of $I$ on the lengths of closed geodesics on $S'$. Recall from the previous part that the length of $\delta_j$ determines the twist parameter $\alpha_j$ corresponding to $\gamma_j$ up to just two possible choices.

\begin{prop}
	Allow $S'$ to be a hyperbolic surface with Bers' constant $B$ and pants decomposition $\gamma_1,...,\gamma_m$ with $l(\gamma_j) \leq B$ for all $j$, then each of the transversals $\delta_j$ to $\gamma_j$ satisfies $l(\delta_j) \leq 3B - 4\log(I) + 12\log(2)$.
\end{prop}

\begin{proof}
	We first proceed in the case that our surface has no cusps. We will construct a curve homotopic to the the transversal about a geodesic by decomposing the pairs of pants about that geodesic into pairs of isometric geodesic hexagons. In particular, this homotopic curve will be 4 copies of the curve $b$ discussed below, along with two small arcs along the given geodesic.\\ \tab
	Let $Y$ be a pair of pants in the chosen pants decomposition of $S$ with boundary geodesics $\gamma_i, \gamma',$ and $\gamma''$. Then for each pair of boundary components there exists a unique geodesic intersecting both of them perpendicularly. These three segments decompose $Y$ into two isometric right-angled hexagons. Adding the common perpendicular $b$ from $\gamma_j$ to the opposite side in both hexagons decomposes $Y$ into four right-angled pentagons. As the hexagons were isometric, the set of four pentagons contains two isometric pairs. Thus, $\gamma_j$ is split into four arcs of two different lengths. Let $a$ be one of these segments on $\gamma_j$ with length greater than $\frac{1}{4}l(\gamma_j)$. \\ \tab
	\begin{figure}[H]
		\begin{center}
			\includegraphics[width=0.5\textwidth]{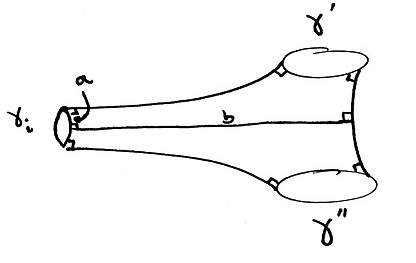} 
		\end{center}
		\caption{}
	\end{figure}
	If we arrange so that things are labeled as in the diagram, then a standard fact from the hyperbolic trigonometry of right-angled pentagons (Theorem 2.3.4 in \cite{Buser2}) yields that
		\begin{equation*}
			\cosh(\frac{1}{2} l(\gamma')) = \sinh(a)\sinh(b).
		\end{equation*}
	Or, using our Bers' constant, and noting $a \geq \frac{1}{4}l(\gamma_j)$ and $b \leq B/2$ we have that 
		\begin{align*}
			b &\leq \text{arcsinh} \left( \frac{\cosh(B/2)}{\sinh(\frac{1}{4}l(\gamma_j))} \right)\\
			&\leq \log \left(\frac{8e^{B/2}}{I}\right) = \frac{B}{2} - I + 3\log(2)
		\end{align*}
	by approximating $\text{arcsinh}(x) < \log(2x), \cosh(x) < e^x,$ and $x < \sinh(x)$. The perpendicular on the opposite side of $Y$ has the same length as $b$ because the pentagons on that side are isometric to those on the front side. Note that $b$ depends only on the lengths of the boundary components on the pair of pants in question. As the two boundary lengths opposite $\gamma_j$ were bounded by the Bers' constant, we see that same approximation on the other pair of pants meeting $\gamma_j$ yields the same bound. Now, the transversal to $\gamma_j$ lies in the free homotopy class of the curve formed by taking these four altitudes and two small arcs on $\gamma_j$ joining the appropriate ends of these copies of $b$. By the constraint placed on the twist parameters, these two smaller arcs have total length less than $l(\gamma_j)$. Thus, the length of the transversal $\delta_j$ is bounded above by
	\begin{equation*}
		l(\delta_j) \leq l(\gamma_j) + 4b \leq B + 4(\frac{B}{2} - I + 3\log(2)) = 3B - 4I + 12\log(2).
	\end{equation*}
\tab Now suppose that the surface $S$ has a cusps. We again look a pair of pants $Y$ with boundary geodesics $\gamma_k$ and $\gamma'$ so that $Y$ has a cusp. Then for both $\gamma_k$ and $\gamma'$, there exists a unique perpendicular geodesic which goes out the cusp, $\xi_k$ and $\xi'$ respectively. Adding the common perpendiculars $c$ from $\gamma_k$ to $\gamma'$ and $b$ from $\gamma_k$ to $\xi'$, we decompose the front half of $Y$ into a geodesic pentagon and a geodesic trirectangle with one ideal point. As $Y$ is made from the pasting of two isometric degenerate hexagons, we have an isometric construction on the opposite side of $Y$. Allow $a$ to be the segment on $\gamma_k$ of length $\geq \frac{1}{4}l(\gamma_k)$ on the front half of $Y$. If $a$ lies in the geodesic pentagon, then we can proceed exactly as above. If $a$ lies in the trirectangle, however, then we can use the standard trigonometric fact \cite[Thm 2.3.1 (i)]{Buser2} to conclude that 
	\begin{equation*}
		1 = \cos(0) = \text{sinh}(a)\text{sinh}(b) \Rightarrow b = \text{arcsinh}{\frac{1}{\text{sinh}(a)}}
	\end{equation*}
As $1 \leq \text{cosh}(\alpha)$ for all $\alpha$, we have that
	\begin{equation*}
		b \leq \text{arcsinh}{\frac{\text{cosh}(B/2)}{\text{sinh}(a)}}
	\end{equation*}
and we can bound the length of the transversal via the same procedure as above. For two cusps, we just create the analogous pair of geodesic trirectangles on the surface and arrive at the same bound.
	\begin{figure}
		\begin{center}
			\includegraphics[width=0.5\textwidth]{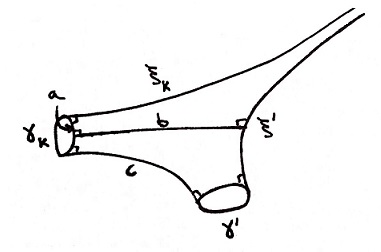} 
		\end{center}
		\caption{}
	\end{figure}
\end{proof}

As $S$ and $S'$ are length isospectral, the lengths of the transversals bounded in the previous proposition must appear in the spectrum of $S$. By the estimate in the previous section on length spectrum growth rate (Proposition 4.3), there are at most $(2g - 2 + n)\exp(3B - 4I + 12\log(2) + 6)$ possible lengths in this portion of the spectrum and hence this many choices for the length of each transversal on $S'$. As the length of the transversal at a geodesic in our pants decomposition determines the twist parameter up to 2 possible values, we obtain the following bound after accounting for the $3g-3+n$ geodesics in the pants decomposition:

\begin{lem}
	The number of possible twist parameters for a surface length isospectral to a hyperbolic surface $S$ with Bers' constant $B$ and lower bound on systole length $I$ based on graph $G$  and with fixed length parameters $L$ is bounded above by $2(\frac{1}{2}(2g - 2 + n))^{3g-3+n}\exp\{(3g-3+n)(3B - 4I + 12\log(2) + 6)\}$.
\end{lem}

Multiplying together the bounds for the graphs and length and twist parameters, we obtain the desired result.

\begin{rem}
The preceding proof is adapted from the first half of the proof of Theorem 1.1 in Buser \cite{Buser2} and arose from attempts to refine this proof. While there is room to tighten the constants used in Buser's approximations, any significant improvement to the asympototics of the bound seems unlikely with his technology unless the asymptotics on the Bers' constant for closed surfaces is improved. The proof has a factor of $e^{3g\log(g)}$ built in to account for the different graphs corresponding to the pants decompositions, rendering improvement beyond these asymptotics impossible. The factor of the genus squared in the exponent of Buser's bound comes from the same process described above: the Bers' constant depends on the genus, the length spectrum growth rate exponentiates this constant, and accounting for the number of geodesics in the pants decomposition introduces another factor of the genus. Thus, improvement in the Bers' constant results directly in the improvement of the bound.\\ \tab
	The bound produced here is heavily dependent upon our bound on systole length.  It is a consequence of the collar theorem that as a simple closed geodesic shrinks, it has an embedded cylinder of increasing length about it on the surface. This results in the lengthening of the transverse curve to the geodesic. Thus, allowing for arbitrarily short geodesics results in an unbounded collection of transverse curves and the proof presented here is no longer valid. Running Buser's proof for the general case with just the improved Bers' constant would only improve the coefficient of $g^2$ to about 679 due to the complications arising from the short geodesics.
\end{rem}
	
\section{Acknowledgements}

This paper was written during the Mathematics Research Experience for Undergraduates at the University of Michigan. I would like to thank my advisors Richard Canary and Benjamin Linowitz for their patience and insight, as well as the University of Michigan Mathematics Department for running such a rewarding REU program.

\end{document}